\documentclass[a4paper,11pt]{article}

\usepackage{etoolbox}
\newcommand{\zerodisplayskips}{%
  \setlength{\abovedisplayskip}{4.5pt}%
  \setlength{\belowdisplayskip}{4.5pt}%
  \setlength{\abovedisplayshortskip}{4.5pt}%
  \setlength{\belowdisplayshortskip}{4.5pt}}
\appto{\normalsize}{\zerodisplayskips}
\appto{\small}{\zerodisplayskips}
\appto{\footnotesize}{\zerodisplayskips}
\usepackage[margin=2.5cm]{geometry}

\usepackage[latin1]{inputenc}
\usepackage{mathrsfs,amssymb,amsmath,amsthm}
\usepackage{amsfonts}
\usepackage{amssymb}
\usepackage{amsfonts}
\usepackage{dsfont}
\usepackage{enumitem}

\usepackage[usenames,dvipsnames]{color}
\usepackage{tikz}

\usetikzlibrary{arrows,shapes,trees,patterns,backgrounds,shadows,chains,matrix,decorations.pathreplacing,automata}
\usepackage{float}
\usetikzlibrary{calc}

\definecolor{vert}{RGB}{0,128,0}
\definecolor{rouge}{RGB}{255,0,0}
\definecolor{bleu}{RGB}{0,0,255}
\definecolor{orange}{RGB}{255,80,0}

\usepackage{graphicx}
\DeclareGraphicsExtensions{.jpg,.pdf,.png,.gif,.eps}

\newtheorem{theorem}{Theorem}
\newtheorem{definition}{Definition}
\newtheorem{proposition}{Proposition}

\newtheorem{lemma}{Lemma}
\newtheorem{corollary}{Corollary}

\newcounter{step}

\newcommand{\m}{\mathbb}

\newcommand{\1}{\mathds{1}}

\newcommand{\E}{\ensuremath{\mathds{E}}}

\newcommand{\N}{\ensuremath{\mathds{N}}}

\renewcommand{\phi }{\varphi }

\usepackage[textwidth=30mm]{todonotes}

\def \1{\mathbf{1}}
\renewcommand{\P }{\mathds{P}}
\renewcommand{\geq }{\geqslant }
\renewcommand{\leq }{\leqslant }

\DeclareMathOperator*{\bary}{Bar}

\begin{document}
\bibliographystyle{plain}

\title{History-dependent evaluations in POMDPs}
\date{}
\author{Xavier Venel \thanks{CES, PSE-Universit\'e Paris 1 Panth\'eon Sorbonne, Paris. France. Email: xavier.venel@univ-paris1.fr
}, Bruno Ziliotto\thanks{CEREMADE, CNRS, Universit\'e Paris Dauphine, PSL Research Institute, Paris, France.
}}


\maketitle

\begin{abstract}
We consider POMDPs in which the weight of the stage payoff depends on the past sequence of signals and actions occurring in the infinitely repeated problem. We prove that for all $\varepsilon>0$, there exists a strategy that is $\varepsilon$-optimal for any sequence of weights satisfying a property that interprets as ``the decision-maker is patient enough". This unifies and generalizes several results of the literature, and applies notably to POMDPs with limsup payoffs.  
\end{abstract}

%

\underline{Keywords:} Markov decision processes, Partial Observation, Long-run average payoff.\\

\vspace{-0.2cm}
\section*{Introduction}

A Partially Observable Markov Decision Process (POMDP) is a discrete-time dynamic decision model in which at each stage, the decision-maker takes a decision that determines together with the current state a stage payoff. The state follows a controlled Markov chain. The decision-maker may not be informed of the current state, but receives a signal at each stage. Thus, POMDPs generalize the Markov Decision Process (MDP) model, introduced by Bellman \cite{Bellman_57} and studied by Blackwell \cite{B62}, where the decision-maker knows the state. Their study was initiated independently by various authors \cite{aastrom1965optimal}, \cite{drake1962observation}, \cite{shiryaev1966theory}.
POMDPs are closely related to MDPs with Borel space. Indeed, a standard way to analyze them is to consider a MDP on an auxiliary state space, the belief space (see \cite{rhenius74}, \cite{sawaragi70}, \cite{yushkevich76}). The fact that this auxiliary state space is infinite makes the analysis of POMDPs quite delicate. Indeed, even the belief dynamics generated by a ``simple" strategy (e.g., belief-stationary) can be erratic. To recover some regularity on this dynamics, many papers consider some irreducibility-type assumption on the transition function of the POMDP (see e.g. \cite{Arapostathis_93,Borkar_1988,Borkar_2000}). In this paper, we will not need this type of assumption, and consider the POMDP model in its full generality.      \\
This paper concerns POMDPs with long duration, where state, action and signal sets are finite. Considerable work has been devoted to the definition and characterization of the optimal long-term payoff (see \cite{Arapostathis_93}). 
We emphasize below several results of the literature with regard to this question, that are intimately related to our paper:

\begin{itemize}
\item {\em $n$-stage problem and discounted problem.}
 A first standard approach is to consider the $n$-stage problem, with average expected payoff $\m{E}(\frac{1}{n}\sum_{m=1}^n r_m)$, where $r_m$ stands for the payoff at stage $m$. The long-term optimal payoff can be defined as the limit of the $n$-stage value, as $n$ tends to infinity (asymptotic value). By Rosenberg, Solan and Vieille \cite{RSV02}, this limit exists and coincides with the limit of the discounted value, as the discount factor vanishes. 

\item {\em Uniform value.}
 Another standard approach is to define the payoff in the infinite problem as being the 
long-run average payoff criterion $\liminf_{n \rightarrow +\infty} \m{E}(\frac{1}{n}\sum_{m=1}^n r_m)$ (see Arapostathis et al. \cite{Arapostathis_93}), and consider the value of this problem (\textit{uniform value}). Again, by Rosenberg, Solan and Vieille \cite{RSV02}, this coincides with the asymptotic value.

\item {\em Liminf (average) evaluation.}
 A third approach is to define the payoff in the infinite problem as being  $ \m{E}(\liminf_{n \rightarrow +\infty}\frac{1}{n}\sum_{m=1}^n r_m)$, as studied in Gillette [14]. The authors \cite{VZ16} have proved that the value of this problem coincides with the asymptotic value. 

\item {\em General uniform value.}
 As an extension of the first approach, Renault and Venel \cite{RV17} consider problems with weighted payoff: the payoff is $\m{E}\left(\sum_{m \geq 1} \theta_m r_m\right)$, where the $\theta_m$ are positive numbers that sum to 1. They define the impatience of $\theta$ as being the total variation $I(\theta)=\sum_{m \geq 1} | \theta_{m+1}-\theta_m|$. They proved that the value of the problem where the total payoff is the inferior limit of $\m{E}\left(\sum_{m \geq 1} \theta_m r_m\right)$, as $I(\theta)$ tends to 0, is again the asymptotic value. 
\end{itemize}
The first and second result can be deduced either from the third or from the fourth. Though the third and fourth result are seemingly independent, the following connection can be made:
one can see the liminf evaluation as a payoff with random weights, where the weights depend on the history. This observation was the starting point of this paper.  Since these four notions of value coincide in finite POMDPs, we will simply call them the {\em asymptotic value} in the following.

 \paragraph{Contribution of the paper}
This paper unifies and generalizes the above results, and aims at finding the largest possible class of optimal long-term payoffs that coincides with the asymptotic value. For this purpose, we introduce the problem with history-dependent evaluations. As in \cite{RV17}, we consider weighted payoffs, but we allow $\theta_m$ to depend on the past observed history of the problem: that is, the sequence of actions and signals that are generated before stage $m-1$ (included). 
We define the \textit{irregularity} of an evaluation as being  $\E_{x_1}^\sigma \left(|\theta_1|+\sum_{m=1}^{+\infty} |\theta_m- \theta_{m+1}|\right)$. 
The first main result of this paper is to prove that the value of the problem where the total payoff is the inferior limit of $\m{E}\left(\sum_{m \geq 1} \theta_m r_m\right)$, as the irregularity tends to 0, is equal to the asymptotic value. This generalizes all the results mentioned previously.
Moreover, we provide several classes of weighted payoffs for which none of the previous results applied. We give several examples that illustrate the tightness of our assumptions. 
 
In addition, our result has the following significant consequence: the value of the problem with 
total payoff 
$\m{E}\left(\limsup_{n \rightarrow +\infty} \frac{1}{n} \sum_{m=1}^n g_m \right)$ is also the asymptotic value, where $g_m$ is the expectation of the stage payoff with respect to the belief at stage $m$. 
The main difficulty for deducing this result is that the limsup depends on the behavior in the long run, hence the relative weight of the payoff at stage $m$ may depend on the future. 

A straightforward consequence is that, when the decision-maker is informed of his payoffs (\textit{POMDP with known payoffs}), then the value of the problem with total payoff
\\
$\m{E}\left(\limsup_{n \rightarrow +\infty} \frac{1}{n} \sum_{m=1}^n r_m\right)$ is also equal to the asymptotic value. Last, we present an example that shows that in the absence of the known-payoff assumption, the value of the above problem may be strictly greater than the asymptotic value. 

Section 1 states the model and main definitions, and Section 2 presents our contributions. Section 3 gives a sketch of the main results' proofs, and compare them with previous literature techniques. Sections 4 and 5 prove the two theorems.

\section{Model}
\subsection{POMDPs}

Let us start with a few notations.
We denote by $\m{N}^*$ the set of strictly positive integers. If $A$ is a measurable space, we denote by $\Delta(A)$ the set of probability measures over $A$. If $a\in A$, we denote by $\delta_a$ the Dirac mass at $a$. If $(A,d)$ is a compact metric space, we will always equip $(A,d)$ with the Borel $\sigma$-algebra, and denote by $\mathcal{B}(A)$ the set of Borel subsets of $A$. The set of continuous functions from $A$ to $[-1,1]$ is denoted by $\mathcal{C}(A,[-1,1])$. The set $\Delta(A)$ is compact metric for the Kantorovich-Rubinstein distance $d_{KR}$, which metrizes the weak$^*$ topology. Recall that the distance $d_{KR}$ is defined for all $z$ and $z'$ in $\Delta(A)$ by
\[
d_{KR}(z,z'):=\sup_{f \in E_1} \left|\int_A f(x) z(dx)-\int_A f(x) z'(dx) \right|=\inf_{\pi \in \Pi(z,z')} \int_{A\times A} d(x,y) \pi(dx,dy),
\]
where $E_1\subset \mathcal{C}(A,[-1,1])$ is the set of $1$-Lipschitz functions from $(A,d)$ to $[-1,1]$ and $\Pi(z,z') \subset \Delta(A \times A)$ is the set of measures on $A \times A$ with first marginal $z$ and second marginal $z'$. Because $A$ is compact, the infimum is a minimum. When $A$ is finite and $d$ is the discrete metric, then $d_{KR}$ is the $L^1$-norm on $\Delta(A)$. For $f \in \mathcal{C}(A,[-1,1])$, the linear extension of $f$ is the function $\hat{f} \in \mathcal{C}(\Delta(A),[-1,1])$, defined for $z \in \Delta(A)$ by
\begin{equation*}
\hat{f}(z):=\int_{A} f(x) z(dx).
\end{equation*}

A (finite) \textit{Partially Observable Markov Decision Process} $\Gamma=(K,I,S,q,r)$ is defined by the following elements:
\begin{itemize}
\item $K$ is a finite \textit{state space},
\item $I$ is a finite \textit{action space},
\item $S$ is a finite \textit{signal space},
\item $q: K\times I \rightarrow \Delta(K \times S)$ is the \textit{transition} function that associates to each pair (state, action) a distribution over states and signals,
\item $r: K \times I \rightarrow [0,1]$ is the \textit{payoff function}.
\end{itemize}
Throughout the paper, we use the notation $X:=\Delta(K)$. 
This article only considers POMDPs defined by finite sets, and therefore the finiteness assumption will not be stated anymore. 

\noindent Given $x_1 \in X$, the \textit{POMDP starting from $x_1$},  denoted by $\Gamma(x_1)$, proceeds in the following way:
\begin{itemize}
\item
Before the game starts, an initial state $k_1$ is drawn according to the distribution $x_1$. The decision-maker knows $x_1$ but not $k_1$.
\item
At each stage $m \geq 1$, the decision-maker chooses $i_m \in I$, and receives a payoff $r(k_m,i_m)$. A pair $(k_{m+1},s_m)$ is drawn according to $q(k_m,i_m)$. The next state is $k_{m+1}$, and the decision-maker is informed of $s_m$. 
\end{itemize}
The sequence $(k_1,i_1,s_1,\dots,k_m,i_m,s_m,\dots)$ is called a \textit{play}. The set of plays is denoted by $H_\infty=(K\times I \times S)^{\m{N}}$. For every $n\geq 1$, we define the set of \textit{histories of length $n$} by $H_n=(K\times I \times S)^{n-1}\times K$ and the set of \textit{histories} by $H_f=\cup_{n\geq 1} H_n$.\\

We denote by $H^o_\infty=(I \times S)^{\m{N}}$ the set of \textit{observed plays}. For every $n\geq 1$, we define the set of \textit{observed histories at stage $n$} by $H^o_n=(I \times S)^{n-1}$ and the set of \textit{observed histories} by $H^o_f=\cup_{n\geq 1} H^o_n$.\\

\begin{definition}
A {\em behavior strategy} $\sigma$ of the decision-maker is a function from $H^o_f$ to $\Delta(I)$. It is said to be {\em pure} if for every $h\in H^o_f$, $\sigma(h)$ is a Dirac mass at some action $i\in I$. The set of pure strategies is denoted by $\Sigma_p$. \\
\end{definition}
\begin{definition}
A pure strategy $\sigma$ is said to have \emph{finite memory} if it can be modeled by a finite-state transducer. Formally, such a strategy is described as $\sigma=(\sigma_u,\sigma_a,M,m_0)$, where $M$ is a finite set of memory states, $m_0$ is the initial memory state, $\sigma_a:M \rightarrow I$ is the action selection function and $\sigma_u:M\times I \times S \rightarrow M$ is the memory update function.
\end{definition}

\subsection{History-dependent evaluation}

Throughout the paper, we identify a history $h_m$ and the subset of plays that have prefix $h_m$. This subset of plays is called a \emph{finite cylinder}. We denote by $\mathcal{F}$ the $\sigma$-field on $H_\infty$ generated by finite cylinders. An initial probability $x_1 \in X$ and a behavior strategy $\sigma$ naturally induce a probability measure on the set of finite cylinders, that extends in a unique way to $H_{\infty}$, by the Kolmogorov extension theorem. This probability measure is denoted by $\m{P}_{x_1}^{\sigma}$ and the expectation under $\m{P}$ is denoted by $\E_{x_1}^\sigma$.\\

Similarly, an observed play and an observed history can be seen as subsets of plays. We denote by $\mathcal{F}_m^o$ the $\sigma$-field on $H_\infty$ generated by observed histories of length $m$ and by $\mathcal{F}^o$ the $\sigma$-field generated by observed histories.

\begin{definition} \label{def:eval}
An \emph{evaluation} is a sequence of functions $\theta=(\theta_m)_{m\geq 1}$ from $H_\infty$ to $[0,1]$. It is said to be
\begin{itemize}
\item
\emph{history-dependent} if for every $m\geq 1$, $\theta_m$ is measurable with respect to $\mathcal{F}^o_m$.
\item \emph{normalized} if for every $h_\infty \in H_\infty$, 
$
\displaystyle \sum_{m=1}^{+\infty} \theta_m(h_\infty)=1.
$
\item \emph{normalized in expectation at $(x_1,\sigma)$} if 
$
\displaystyle
\E_{x_1}^\sigma\left(\sum_{m=1}^{+\infty} \theta_m(h_\infty) \right)=1.
$
\end{itemize}
We denote by $\Theta$ the set of evaluations and by $\Theta_{\mathcal{N}}$ the set of history-dependent normalized evaluations.
\end{definition}


%
%
%

Let us compare our definition with literature. Two standard evaluations are the $n$-stage evaluations and the $\lambda$-discounted evaluations. They correspond to the case where for all $m \geq 1$, $\theta_m$ is deterministic, and $\theta_m=\frac{1}{n} 1_{m \leq n}$ for the former, and $\theta_m=\lambda(1-\lambda)^{m-1}$ for the latter. Renault and Venel \cite{RV17} consider more general evaluations, but that are still deterministic. 

Neyman and Sorin \cite{NS10} consider $n$-stage evaluations, in which $n$ is a random variable. This is not a particular case of our model, since the latter features evaluations which are deterministic functions of the past observed history, while the random variable $n$ in the former follows a history-independent process. 
 \\

%
%

Let $x_1 \in X$ and $\theta \in \Theta$. The problem starting from $x_1$ and with evaluation $\theta$, denoted by $\Gamma_{\theta}(x_1)$, has payoff function:
\begin{align*}
\forall \sigma \in \Sigma, \quad \gamma_\theta(x_1,\sigma) &:=\E^{\sigma}_{x_1}\left( \sum_{m=1}^{+\infty} \theta_m r(k_m,i_m) \right).
\end{align*}
The value $v_\theta(x_1)$ of this problem is the maximum expected payoff with respect to behavior strategies:
\begin{align}\label{vn}
v_\theta(x_1):=\sup_{\sigma \in \Sigma} \gamma_\theta(x_1,\sigma)=\sup_{\sigma \in \Sigma_p} \gamma_\theta(x_1,\sigma).
\end{align}
The fact that the supremum can be taken over pure strategies is a consequence of Feinberg \cite[Theorem 5.2]{F96}.  For every $\varepsilon>0$, a strategy $\sigma$ is said to be {\em $\varepsilon$-optimal in $\Gamma_\theta(x_1)$} if
\[
\gamma_{\theta}(x_1,\sigma) \geq v_{\theta}(x_1)-\varepsilon.
\] 
In particular, when $\theta_m=\frac{1}{n} 1_{m \leq n}$ is the deterministic $n$-stage evaluation ($n \geq 1$), the corresponding value is denoted by $v_n(x_1)$. Moreover, the following result is known:

\begin{proposition}
[Rosenberg, Solan and Vieille \cite{RSV02}] Let $\Gamma$ be a POMDP and $x_1 \in X$ an initial distribution. The sequence $(v_n(x_1))_{n\geq 1}$ converges to some limit $v^*(x_1)$, called the {\em asymptotic value}.
\end{proposition}

\section{Contributions}
\subsection{Weighted value}

The first definition generalizes the impatience introduced in Renault and Venel \cite{RV17} to history-dependent evaluations. 

\begin{definition}
For every evaluation $\theta$, every $x_1\in X$ and every strategy $\sigma$, we define $I(\theta,x_1,\sigma)$ in $[0,+\infty]$ by
\[
I(\theta,x_1,\sigma)= \E_{x_1}^\sigma \left(|\theta_1|+\sum_{m=1}^{+\infty} |\theta_m- \theta_{m+1}|\right),
\]
and the \emph{irregularity} by
\[
I(\theta,x_1)=\sup_{\sigma \in \Sigma} I(\theta,x_1,\sigma).
\]
\end{definition}

The following definition adapts the classical notion of uniform value to history-dependent normalized evaluations. 

%
%


\begin{definition}
Let $x_1 \in X$. The POMDP $\Gamma(x_1)$ has a \emph{weighted value} $v_{\infty}(x_1) \in [0,1]$ if 
\begin{itemize}
\item for all $\varepsilon>0$, there exists $\alpha>0$ such that for all $\theta \in \Theta_{\mathcal{N}}$,
\[
I(\theta,x_1) \leq \alpha \text{ implies that } v_{\theta}(x_1) \leq  v_\infty(x_1)+\varepsilon.
\]

\item for all $\varepsilon>0$, there exists $\alpha>0$ and $\sigma^*\in \Sigma$ such that:
\begin{align}\label{vu}
 \forall \, \theta \in \Theta_{\mathcal{N}} \  (I(\theta,x_1)<\alpha) \text{ implies that } \gamma_\theta(x_1,\sigma^*) \geq v_\infty(x_1)-\varepsilon.
\end{align}
\noindent A strategy that satisfies the above condition is called an \emph{$\varepsilon$-weighted-optimal strategy}. 
\end{itemize}
\end{definition}

\begin{theorem}\label{uniform}
Let $\Gamma$ be a POMDP. For every $x_1\in X$, $\Gamma(x_1)$ has a weighted value equal to the asymptotic value $v^*(x_1)$.  Moreover, for every $\varepsilon>0$, the decision-maker has an  $\varepsilon$-weighted-optimal pure strategy with a finite memory.
\end{theorem}

Let us emphasize several classes of evaluations that are covered by this result. First, we recover two results that were already known.
\begin{itemize}
\item 
$\lambda$-discounted evaluations and $n$-stage evaluations: when restricted to these evaluations, the weighted value definition coincides with the notion of uniform value. Existence of uniform value has been 
proven by Rosenberg, Solan and Vieille \cite{RSV02} (see also \cite{R11} and \cite{VZ16}). 
\vspace{-0.2cm}
\item
Deterministic evaluations: the weighted value coincides with the notion of general uniform value, which existence has been proven in \cite{RV17}.
\end{itemize}
Moreover, we can study new classes of evaluations:
\begin{itemize}
\item Decreasing evaluations: this corresponds to the case where $(\theta_m)$ is almost surely decreasing, and $\theta_1$ vanishes. Note that the usual ``trick" to 
express decreasing evaluations as a convex combination of average payoffs would not work here, due to the fact that the weights depend on the history. 
\vspace{-0.2cm}
\item
$N$-piecewise constant evaluations: this corresponds to the case where there exists some constant $N$ such that almost surely, $(\theta_m)$ takes at most $N$ different values, and $\sup_{m \geq 1} \m{E}(\theta_m)$ vanishes. 
\end{itemize}


We now give two examples that show that the measurability assumption of $\theta_m$ with respect to $\mathcal{F}^0_m$ is essential. 
\begin{proposition} \label{ex1}
There exists a POMDP $\Gamma$,  an initial belief distribution $x_1$ and a sequence of evaluations $(\theta^l)$ such that for all $l\geq 1$ and $m\geq 1$, $\theta^l_m$ is measurable with respect to $\mathcal{F}_m$ (but not to $\mathcal{F}^0_m)$, the impatience of $(\theta^l)$ vanishes, and 
\begin{equation*}
v^*(x_1)<\lim_{l \rightarrow +\infty} v_{\theta^l}(x_1).
\end{equation*}
\end{proposition}
\begin{proof}
We consider the following example with two states $K=\{\alpha,\beta\}$, two actions $I=\{\alpha,\beta\}$ and one signal $S=\{s_0\}$. The payoff is $1$ if state and action match, and 0 otherwise. Moreover, the state never changes. \\
Let $x_1:=\frac{1}{2} \cdot \delta_{\alpha}+\frac{1}{2} \cdot \delta_{\beta}$. For any $\theta \in \Theta_{\mathcal{N}}$, any strategy in $\Gamma_{\theta}(x_1)$ gives payoff $1/2$, thus \\
$v_{\theta}(x_1)=1/2$. 
\\
Now consider the following evaluation $\theta^l$ ($l \geq 1$) which depends only on the state variable at the initial stage:
\begin{itemize}
\item if $k_1=\alpha$ then the weight is concentrated on the first $l$ stages: for every $1\leq m \leq l$, $\theta^l_m(\alpha,...)=\frac{1}{l}$ and for every $m\geq l+1$, $\theta^l_m(\alpha,...)=0$.
\item if $k_1=\beta$ then the weight is concentrated between stage $l+1$ and stage $2l$: for every $l+1\leq m\leq 2l$, $\theta^l_m=\frac{1}{l}$ and for every $m\leq l$ or $m\geq 2l+1$, $\theta^l_m=0$.
\end{itemize}
Consider the strategy that plays action $\alpha$ during $l$ stages and then $\beta$ forever. This strategy yields the maximal payoff $1$ in $\Gamma_{\theta^l}(x_1)$, hence $v_{\theta^l}(x_1)=1$. Since $I(\theta^l_m,x_1)=2/l$, the result follows. 
\end{proof}

%
\begin{proposition} \label{ex2}
There exists a POMDP $\Gamma$,  an initial belief distribution $x_1$ and a sequence of evaluations $(\theta^l_m)$ such that for all $l\geq 1$ and $m\geq 1$, $\theta^l_m$is measurable with respect to $\mathcal{F}^0$ (but not with respect to $\mathcal{F}^0_m$), the impatience of $(\theta^l)$ vanishes, and
\begin{equation*}
v^*(x_1)<\lim_{l \rightarrow +\infty} v_{\theta^l}(x_1).
\end{equation*}
\end{proposition}
\begin{proof}
Consider the Markov chain on $K=\{\alpha,\beta\}$, where at each stage, the state is drawn uniformly. The payoff is 1 in state $\alpha$, and the payoff is 0 in state $\beta$. Let $x_1:= \frac{1}{2} \cdot \delta_{\alpha}+\frac{1}{2} \cdot \delta_{\beta}$. Thus, 
$v_{\infty}(x_1)=1/2$. Consider the evaluation $\theta^l_m$ such that: 
$\theta^l_m=1/l$ for $N+1 \leq m \leq N+l$, where $N$ is the first stage such that the state is $\alpha$ in all stages between $N+1$ and $N+l$, and $\theta^l_m=0$ otherwise. We have $v_\theta(x_1)=1$ and $I(\theta^l_m,x_1)=2/l$, and thus the result is proved. 
\end{proof}
\subsection{Limsup evaluations}
In addition to the above classes, it turns out that the proof of Theorem \ref{uniform} allows us to deal with another type of evaluations, namely the $\limsup$ (average) evaluations. These are not history-dependent evaluations in the sense of Definition \ref{def:eval}, but they can be approximated by them in some sense. We turn to this point now. 
\\
In a previous paper \cite{VZ16}, we focused on the $\liminf$ evaluation, where informally the decision-maker is pessimistic about the length of the game. Given an initial belief $x_1 \in X$, the \textit{POMDP with $\liminf$ evaluation} $\underline{\underline{\Gamma}}_{\infty}(x_1)$ is the problem with strategy set $\Sigma$, and payoff function $\underline{\underline{\gamma}}_{\infty}$ defined for all $\sigma \in \Sigma$ by
\begin{equation*}
\underline{\underline{\gamma}}_{\infty}(x_1,\sigma):=\E^{\sigma}_{x_1}\left(\liminf_{n \rightarrow +\infty} \frac{1}{n} \sum_{m=1}^n r(k_m,i_m) \right).
\end{equation*}
This type of payoff was introduced by Gillette for stochastic games \cite{gillette_57}. The value of $\underline{\underline{\Gamma}}_{\infty}(x_1)$, called the \textit{$\liminf$ value}, is
\begin{equation}\label{w}
\underline{\underline{v}}_\infty(x_1):=\sup_{\sigma \in \Sigma} \underline{\underline{\gamma}}_{\infty}(x_1,\sigma)=\sup_{\sigma \in \Sigma_p} \underline{\underline{\gamma}}_{\infty}(x_1,\sigma).
\end{equation}
The supremum  can be taken over pure strategies as a direct consequence of Theorem 5.2 in Feinberg \cite{F96}. This value coincides with the asymptotic value $v^*(x_1)$ \cite{VZ16}, and for all $\varepsilon>0$, there exists an $\varepsilon$-optimal pure strategy with finite memory \cite{CSZ19}. \\

 Let us now analyze the opposite situation where the decision-maker is optimistic about the length of the game. This leads to consider the \textit{$\limsup$-evaluation}, where the payoff is defined by
\begin{equation*}
\overline{\overline{\gamma}}_{\infty}(x_1,\sigma)=\E^{\sigma}_{x_1}\left(\limsup_{n \rightarrow +\infty} \frac{1}{n} \sum_{m=1}^n r(k_m,i_m) \right).
\end{equation*}
The associated value is denoted by $\overline{\overline{v}}_{\infty}(x_1)$, and called the $\limsup$ value.
Unfortunately, this evaluation is not a function of the observed history, which makes it outside of the scope of Theorem \ref{uniform}. By properties of finite Markov chains, if the decision-maker restricts himself to finite-memory strategies, then the $\liminf$ evaluation and the $\limsup$ evaluation coincide but it may not be the case if the strategy does not have finite memory. Though for the $\liminf$ evaluation, infinite-memory strategies does not yield a higher payoff than finite-memory strategies, this is not the case for the $\limsup$ evaluation, as shown by the following example.
 
\begin{proposition}\label{strictly}
There exists a POMDP $\Gamma$ and an initial belief distribution $x_1$ such that
\[
v_\infty(x_1) < \overline{\overline{v}}_\infty(x_1).
\]
\end{proposition}
\begin{proof}
We consider the following example with two states $K=\{\alpha,\beta\}$, two actions $I=\{T,B\}$ and only one signal $S=\{s_0\}$ (\textit{blind MDP}). The payoff is $0$ in state $\alpha$ and $1$ in state $\beta$. Moreover, we assume that the transition is given as follows: when the decision-maker plays $T$, the state stays the same, and when he plays $B$, the state switches to the other state. 

Let us consider the initial distribution $(1/2,1/2)$, then this problem admits a weighted value equal to $1/2$, that is guaranteed by any strategy. 

On the contrary, the $\limsup$ value is equal to $1$. The following strategy guarantees $1$ for the $\limsup$ evaluation: 
play $T$ for $2$ stages, play $B$ once, play $T$ for $2^{2^2}$ stages, play $B$ once, play $T$ for $2^{n^2}$ stages, play $B$, and so on and so forth.
\\
Starting from state $\alpha$ (resp. $\beta$), the sequence of states is uniquely determined and payoffs alternate between long blocks of $0$ and long blocks of $1$. Since the block sizes get larger and larger, one can check that the payoff under the $\limsup$ evaluation is equal to $1$ on each of the two infinite histories. Hence, the $limsup$ value is equal to $1$. Notice that the strategy has infinite memory and that its payoff under the $\liminf$ evaluation is equal to $0$.
\end{proof}

There is another way to define such an optimistic evaluation that will lead to a positive result, namely the {\it $\limsup$-belief} evaluation. Given an observed play $h^o_\infty$, we can define $x_m$ as the belief of the decision-maker at stage $m$, conditionally on the observed play: 
\[
x_m=
\begin{cases}
\P_{x_1}^\sigma\left(k_m=.\left|h^o_m\right.\right) & \text { if } \m{P}_{x_1}^{\sigma}(h^0_m)>0, \\
\delta_{k_0} & \text { otherwise},
\end{cases}
\]
where $k_0$ is a fixed arbitrary state. It is a function of $h^o_m$, the observed history until stage $m$, and of the initial belief $x_1$. Moreover, since the decision-maker remembers his actions, this is independent of the strategy used. The {\it $\limsup$-belief} evaluation is then defined by
\begin{equation*}
\overline{\gamma}_{\infty}(x_1,\sigma):=\E^{\sigma}_{x_1}\left(\limsup_{n \rightarrow +\infty} \frac{1}{n} \sum_{m=1}^n g(x_m,i_m) \right),
\end{equation*}
where $g(p,i)=\sum_{k \in K} p(k) g(k,i)$, for all $p \in \Delta(K)$ and $i \in I$. 
Given an initial belief $x_1 \in X$, the \textit{POMDP with $\limsup$-belief evaluation} $\overline{\Gamma}_{\infty}(x_1)$ is the problem with strategy set $\Sigma$, and payoff function $\overline{\gamma}_\infty$. Its value, called the \textit{$\limsup$-belief value}, is denoted by $\overline{v}_\infty(x_1)$.

\begin{theorem}\label{limsup}
Let $\Gamma$ be a finite POMDP. For every $x_1 \in X$, the $\limsup$-belief value $\overline{v}_\infty(x_1)$ and the weighted value $v_{\infty}(x_1)=v^*(x_1)$ coincide.
\end{theorem}


It is interesting to notice the asymmetry between the $\limsup$-belief evaluation and the $\liminf$ evaluation. For the sake of the discussion, let us introduce the $\liminf$-belief evaluation defined by
\begin{equation*}
\underline{\gamma}_{\infty}(x_1,\sigma)=\E^{\sigma}_{x_1}\left(\liminf_{n \rightarrow +\infty} \frac{1}{n} \sum_{m=1}^n g(x_m,i_m) \right).
\end{equation*}
Its value is denoted by $\underline{v}_{\infty}(x_1)$, and called the \textit{$\liminf$-belief value}. 
By definition of the inferior and superior limit and Fatou's Lemma, we clearly have
\begin{equation*}
\underline{\underline{\gamma}}_{\infty}(x_1,\sigma)\leq \underline{\gamma}_{\infty}(x_1,\sigma) \leq \overline{\gamma}_{\infty}(x_1,\sigma)\leq \overline{\overline{\gamma}}_{\infty}(x_1,\sigma),
\end{equation*}
and the same inequalities hold for the corresponding values:
\begin{equation*}
\underline{\underline{v}}_{\infty}(x_1)\leq \underline{v}_{\infty}(x_1) \leq \overline{v}_{\infty}(x_1)\leq \overline{\overline{v}}_{\infty}(x_1),
\end{equation*}
To summarize, we have obtained the following results. The authors \cite{VZ16} showed that $\underline{\underline{v}}_{\infty}(x_1)$ is equal to $v^*(x_1)$. Proposition \ref{strictly} shows that $\overline{\overline{v}}_\infty(x_1)$ may be strictly greater than $v^*(x_1)$. Finally, in Theorem \ref{limsup}, we prove that $\overline{v}_\infty(x_1)$ is also equal to $v^*(x_1)$. \\



%
%
Nonetheless, under the natural assumption that the decision-maker observes his payoffs, Theorem \ref{limsup} implies that the $\limsup$ value coincides with the asymptotic value:
\begin{definition} \label{Cdef1}  \rm A POMDP has {\it  known payoffs} if the set of states $K$ can be partitioned
in a way such that for all states $k$, $k'$, $k_1$, $k_2$ in $K$, actions $i$, $i'$ in $I$, and signal $s$ in $S$:
\begin{itemize}
\item
if $k_1\sim k_2$ then $r(k_1,i)=r(k_2,i)$  (two states in the same element of the partition induce the same payoff function), 
\item
 if $q(k,i)(k_1,s)>0$ and $q(k',i')(k_2,s)>0$ then $k_1\sim k_2$ (observing the public signal is enough to deduce the element of the partition containing the current state).
 \end{itemize}
 \end{definition}

\begin{corollary} \label{aux_MDP}
Assume that the POMDP $\Gamma$ has known payoffs. Then
\[
v_\infty(x_1) = \overline{\overline{v}}_\infty(x_1).
\]
\end{corollary}

\begin{proof}
Consider the auxiliary POMDP $\Gamma'=(K',I,S,q',r')$, such that $K'=K \times g(K \times I)$, and $r'(k,u)=u$. The transition on (first component, signal) is the same as in $\Gamma$, and
the second component of the state at stage $m+1$ corresponds to the stage payoff at stage $m$ in $\Gamma$.  Formally, for every $(k,u)\in K'$ and every $(i,s)\in I\times S$,
\[
q'((k,u),i)=\sum_{(l,s)\in K\times S} q(k,i)(l,s)\delta_{(k,r(k,i)),s}.
\]
Naturally, the sets of observed histories in both games are equal and therefore the sets of strategies are the same. One can check easily that, for each strategy, the $\limsup$ payoff in both games are equal, hence their $\limsup$ values coincide. 
Moreover, the $\limsup$-belief evaluation and the $\limsup$ evaluation coincide in $\Gamma'$. Applying Theorem \ref{limsup} to $\Gamma'$, we obtain the result.
\end{proof}

\section{Sketch of proof and comparison with literature}
In the remainder of the paper, we will assume that the payoff $r$ is only a function of the state variable. Indeed, as explained in the proof of Corollary \ref{aux_MDP}, given a POMDP $\Gamma$, one can build an auxiliary POMDP where the new state space is a finite subset of $K\times [0,1]$, and the payoff only depends on the state, and is shifted one stage onward. Existence of $\limsup$ value is equivalent in both POMDPs, and the same is true for the weighted value. \\
\vspace{-0.3cm}
\subsection{Theorem \ref{uniform}}
Our proof borrows several ingredients from the three papers \cite{VZ16,RV17,CSZ19}. In the sketch of proof below, we emphasize the differences and common points with these works. 
To prove Theorem \ref{uniform}, we need to prove first that the decision-maker can guarantee $v^*(x_1)$, for any $\theta$ such that $I(\theta,x_1)$ is small enough (\textit{lower bound}), and that he can not do better (\textit{upper bound}). 
\vspace{-0.3cm}
\paragraph{Lower bound}
Let us prove that the decision-maker can guarantee $v^*(x_1)$ in $\Gamma_\theta(x_1)$, 
for any $\theta$ such that $I(\theta,x_1)$ is close to 0. We rely on the existence of $\varepsilon$-optimal strategies with finite memory for the $\liminf$-evaluation, proved in \cite{CSZ19}. A crucial point is that under a finite memory strategy $\sigma$, the process $(state, memory \ state, action, signal)$ is a finite Markov chain. This enables to express both the $\liminf$ evaluation payoff and the history-dependent evaluation payoff in terms of the ergodic structure of the Markov chain, from which the result follows. 
Establishing the first expression is straightforward, while the second one is more involved. First, using the definition of impatience, we prove that it is enough to consider a restricted class of $\theta$, that are constant by blocks. We can then relate the history-dependent payoff to a combination of terms of the form:
\begin{equation*}
\E^{\sigma}_{x_1} \left. \left(  \frac{1}{l}\sum_{m=tl+1}^{(t+1)l} r_m \right| \mathcal{F}_{tl+1} \right),
\end{equation*}
where $l$ is a fixed (large) integer, and $t$ is large. Using the underlying Markov chain structure again, we can bound from below the history-dependent payoff in terms of the ergodic structure of the Markov chain.
\vspace{-0.3cm}
\paragraph{Upper bound}
In a second part, we consider $\theta^l$ such that $I(\theta^l,x_1)$ tends to 0, and prove that $\limsup_{l \rightarrow +\infty} v_{\theta^l} \leq v_{\infty}(x_1)$. 
A sequence of actions and a sequence of signals induce a sequence of beliefs of the decision-maker over the state variable. Such a sequence of beliefs $(x_1,...,x_t,...)$ together with a sequence of weights $(\theta_1,...,\theta_t,...)$ can be aggregated into a probability distribution over beliefs such that, informally, for every $t\geq 1$, $x_t$ has measure $\theta_t$. An initial belief $x_1$ and a strategy $\sigma$ generate a probability distribution on plays, hence a distribution over distributions over beliefs, by the previous construction. By considering the barycenter of this distribution over distributions over beliefs, we obtain a probability distribution over beliefs $\mu(x_1,\sigma,\theta^l)$, called the occupation measure. We consider an accumulation point $\mu^*$ of $(\mu(x_1,\sigma,\theta^l))$, and prove that it is an \textit{invariant measure}. Intuitively, an invariant measure can be interpreted as follows: there exists $\sigma^*: \Delta(K) \rightarrow \Delta(I)$ such that when the initial belief $p$ is drawn according to $\mu^*$, and the decision-maker plays action $\sigma^*(p)$, then the belief at stage 2 is distributed according to $\mu^*$. 
For deterministic evaluations, the fact that the accumulation point is an invariant measure was proved in \cite{RV17}, and is rather straightforward. In our random framework, the proof is much more intricate, and we need to dedicate a whole section to it.

The end of the proof builds on three inequalities. First, by definition of $\mu(x_1,\sigma,\theta^l)$ and some regularity properties, we have
\begin{equation*}
\limsup_{l \rightarrow +\infty} v_{\theta^l}(x_1)=\int_{\Delta(K)} g(x) \mu^*(dx).
\end{equation*}
Moreover,  we have
\begin{equation*}
\int_{\Delta(K)} g(x)\mu^*(dx) \leq \int_{\Delta(K)} v^*(x)\mu^*(dx). 
\end{equation*}
This inequality was already used in \cite{RV17}. Last, we have
\begin{equation*}
\int_{\Delta(K)} v^*(x) \mu^*(dx) \leq v^*(x_1).
\end{equation*}
The above inequality is more challenging. In the deterministic case, this corresponds to the classical decreasing property of values along play trajectories, that follows from a standard dynamic programming principle argument. In the random case, one has to use a martingale argument and the optional sampling theorem. Combining the three inequalities yields the desired result.

\subsection{Theorem \ref{limsup}}
A $\limsup$-belief evaluation is not a weighted evaluation but it is possible to approximate the $\limsup$-belief evaluation by a weighted evaluation that depends on the observed play: for every observed play, consider $N$ such that $\frac{1}{N} \sum_{m=1}^N r(x_m)$ is close to $\limsup_{n \rightarrow +\infty} \frac{1}{n} \sum_{m=1}^n r(x_m)$ and define the weights to be equal to $1/N \cdot 1_{m \leq N}$.

The first difficulty is that $N$ is not measurable with respect to the past history, in general. Thus, this yields an evaluation $\theta$ such that $\theta_m$ is not measurable with respect to $\mathcal{F}^0_m$. As enlightened by Proposition \ref{ex2}, this assumption is crucial for Theorem \ref{uniform} to hold. 
The trick is to consider an $\varepsilon$-optimal strategy $\sigma^*$ for the $\liminf$ evaluation, and to define the conditional distribution 
\[
\rho^l_m=\E^{x_1}_{\sigma^*}\left( \theta^l_m |\mathcal{F}^o_m\right).
\]
This restores the measurability assumption but introduces a second difficulty: this evaluation is not normalized at $(x_1,\sigma^*)$, but only normalized in expectation. This property is not enough to apply Theorem \ref{uniform} directly. Fortunately, part of its proof can still be used. Indeed, the proof of the upper bound of Theorem \ref{uniform} only requires $\theta$ to be normalized in expectation (the normalized assumption is required for the lower bound). 
Thus, if one proves that the impatience of $\rho^l$ tends to 0, then we deduce directly that
\begin{equation*}
\overline{v}_\infty(x_1) \leq v^*(x_1)=\underline{\underline{v}}_\infty(x_1).
\end{equation*}
The converse inequality being trivial, this implies the result. Finally, we show that the impatience of $\rho^l$ tends to 0. This requires precise martingale inequalities and is the main difficulty of the proof.
\section{Proof of Theorem \ref{uniform}}
This section is decomposed into four steps. 
First, we introduce several notations that will be used in the remainder of the section. In particular, we reformulate the payoff as a function of the beliefs and introduce the notion of invariant measure that is the key element of our proof. Second, relying on the existence of a finite-memory strategy that is $\varepsilon$-optimal for the $\liminf$ evaluation \cite{CSZ19}, we prove that the decision-maker can guarantee $v^*(x_1)-\varepsilon$ in a uniform sense in the POMDP. In order to do so, we will use that $v^*(x_1)=\underline{\underline{v}}_\infty(x_1)$. In the third section, assuming that some measure $\mu^*$ is an invariant measure of the POMDP, we prove that $v^*(x_1)$ is the maximal payoff that the decision-maker can guarantee. The fourth section is dedicated to the proof that the measure $\mu^*$ is indeed an invariant measure. 

Recall that in the rest of the paper, the payoff function is assumed to be action-independent, and this is without loss of generality. 
\subsection{Preliminaries}
Given an evaluation $\theta \in \Theta$, $x_1 \in \Delta(K)$ and $\sigma$, we can reformulate the payoff in terms of the belief of the decision-maker. Recall that given an observed play $h^o_\infty$, $x_m$ is defined as the belief of the decision-maker at stage $m$ conditional to the observed history until the current stage. We will forget the dependence to $x_1$ which will be fixed and therefore, we obtain for every $m\geq 1$ that $x_m$ is a function from $H^o_\infty$ to $\Delta(X)$. Denote by $g$ the linear extension\footnote{We choose not to use the notation $\hat{r}$ since $r$ will not play any role in the rest of the paper.} of $r$ to $X$:
\[
\forall \ x\in X,  \qquad g(x)=\sum_{k \in K} x(k)r(k).
\]
The function $g$ is Lipschitz on $X$, hence continuous, and
\begin{align}\label{transformation}
\E^{\sigma}_{x_1} \left( \sum_{m=1}^{+\infty} \theta_m  r(k_m) \right)&=\E^{\sigma}_{x_1} \left( \sum_{m=1}^{+\infty} \theta_m  g(x_m) \right).
\end{align}

Let $A$ and $B$ be two compact metric spaces, equipped with their Borelian $\sigma$-field. We now define the notion of projected image. Let $\psi: A \rightarrow B$ be a measurable function and $\nu \in \Delta(A)$, then we recall that \emph{the image $\mu$  of $\nu$ by $\psi$} is the unique probability measure such that for every measurable mapping $f:B \rightarrow[-1,1]$,
\begin{align*}
\int_{b\in B} f(b)\mu(db)=\int_{a\in A} f(\psi(a))\nu(da).
\end{align*}

In our proofs, we will need a slightly different result that combines the image of a measure and the barycenter. First, we need the following definition (see \cite[Chapter 11, section 1.8]{DM11}):
\begin{definition}
Let $\nu \in \Delta(\Delta(B))$. The \emph{barycenter} of $\nu$ is the unique probability measure $\mu=\bary(\nu) \in \Delta(B)$ such that for all $f \in \mathcal{C}(B,[-1,1])$,
\[
\hat{f}(\mu)=\int_{\Delta(B)} \hat{f}(z) \nu(dz).
\]

\end{definition}


We can compose the two previous notions to obtain the following one.
\begin{definition}
Let $\phi: A \rightarrow \Delta(B)$ measurable and $\nu \in \Delta(A)$, then we define \emph{the projected image $\mu$ of $\nu$ by $\phi$} as the barycenter of the image of $\nu$ by $\phi$. This is the unique measure on $B$ such that for every measurable mapping $f:B \rightarrow[-1,1]$,
\begin{align*}
\int_{b\in B} f(b)\mu(db)=\int_{a\in A} \left(\int_{b\in B} f(b)\phi(a)(db) \right) \nu(da). 
\end{align*}
\end{definition}
The characterization is straightforward by combining the characterization of the image of a measure and the characterization of the barycenter.

A measurable mapping $\sigma:X \rightarrow \Delta(I)$ can be interpreted as a \textit{stationary strategy}, that is, a strategy that plays after every history according to the current belief only. 
We now define the notion of image of a measure over $\Delta(X)$ by a stationary strategy and the notion of invariant measure. Recall that $X=\Delta(K)$.

\begin{definition}\label{image}
Let $\mu \in \Delta(X)$ be a probability distribution over beliefs and $\sigma:X\rightarrow\Delta(I)$ be a stationary strategy, we define $\nu$ \emph{the image of $\mu$ by the strategy $\sigma$} as follows:
\begin{itemize}
\item define $\sigma^{\sharp q}: X \rightarrow \Delta(X)$ by
\[
\sigma^{\sharp q}(x)= \sum_{(k,i,s)\in K\times I\times S} x(k)\sigma(x)(i)q(k,i)(s)\cdot \delta_{\overline{q}(x,i,s)},
\]
where $q(k,i)(s)=\sum_{k' \in K} q(k,i)(k',s)$, $\overline{q}(x,i,s)=(\overline{q}(x,i,s)(k))_{k\in K}$ with $\overline{q}(x,i,s)(k)=\frac{q(x,i)(k,s)}{\sum_{k'\in K} q(x,i)(k',s)}$ and $q(x,i)(k,s)=\sum_{k'\in K} x(k')q(k',i)(k,s)$,
\item define $\nu$ as being the projected image of $\mu$ by $\sigma^{\sharp q}$.
\end{itemize}
\end{definition}

\begin{definition}
A measure $\mu \in \Delta(X)$ is an \emph{invariant measure} of $\Gamma$ if there exists $\sigma$ a stationary strategy such that the image of $\mu$ by the strategy $\sigma$ is $\mu$. 
\end{definition}

\subsection{The decision-maker guarantees at least $v^*(x_1)$}

In this section, we will show that for any normalized history-dependent evaluation $\theta$, the decision-maker guarantees the value $v^*(x_1)$ up to an error term that vanishes when the irregularity tends to 0. To this aim, it is enough to prove the following proposition:
\begin{proposition}\label{limit_corollary_0}
Let $\Gamma$ be a POMDP and $x_1\in X$. 
For every $\varepsilon>0$, there exists a strategy $\sigma$ and $l \geq 1$ such that for every history-dependent normalized evaluation $\theta\in \Theta_{\mathcal{N}}$,
\begin{align}\label{control}
\E_{x_1}^{\sigma} \left( \sum_{m=1}^{+\infty} \theta_m r(k_m) \right) & \geq v^*(x_1)-4l I(\theta,x_1,\sigma)-\varepsilon.
\end{align}
\end{proposition}
The key point of the proof is the existence of a pure strategy with finite memory that is $\varepsilon$-optimal in the problem with $\liminf$ evaluation, proved in \cite{CSZ19}. As highlighted before, the $\liminf$ value $\underline{\underline{v}}_\infty(x_1)$ has been shown to be equal to $v^*(x_1)$. We establish the following result for Markov chains:
\begin{lemma}\label{titi}
Consider a finite Markov chain $\chi$ on a state space $U$ and $y_1 \in \Delta(U)$. For all $\epsilon>0$, there exists $l \geq 1$ such that for any normalized history-dependent evaluation $\theta$ defined on $H'_\infty=(U)^{\m{N}}$ and function $f$ defined from $U$ to $[0,1]$, we have
\begin{align}\label{comparison}
\E_{y_1}^{\chi} \left( \sum_{m=1}^{+\infty} \theta_m f(u_m) \right) & \geq \E_{y_1}^{\chi}\left(\liminf_{n\rightarrow +\infty} \frac{1}{n}\sum_{m=1}^n f(u_m) \right)-4l I(\theta,y_1,\chi)-\varepsilon.
\end{align}
\end{lemma}

%
%
%

Proposition \ref{limit_corollary_0} stems from Lemma \ref{titi}. Indeed, consider 
$\varepsilon>0$. Denote by $\sigma=(\sigma_u,\sigma_a,M,m_0)$ a pure strategy with finite memory that is $(\varepsilon/2)$-optimal at $x_1$ for the problem with $\liminf$ evaluation. The transition $q$ together with $\sigma$ induces a finite Markov chain on $K\times M\times I \times S$ with transition function $\chi$ defined by:
\[
\chi(k,m,i,s)=\sum_{k' \in K,s'\in S} q(k,i')(k',s') \cdot \delta_{k',\sigma_u(m,i',s'),i',s'} \text{ where } i'=\sigma_a(m).
\]
Moreover, we have by construction that the probability $\P_{x_1}^\sigma$ on infinite histories induced by $\sigma$ in the POMDP is equal to the marginal on $H_\infty$ of the probability on $(K\times M \times I \times S)^{\m{N}}$ generated by the Markov chain $\chi$.
Applying Lemma \ref{titi} for $\varepsilon/2$ implies Proposition \ref{limit_corollary_0}. 
\\

The rest of this section is dedicated to the proof of Lemma \ref{titi}.
The state space $U$ being finite, we know that the set of states can be decomposed into ergodic classes $\{U_1,...,U_D\}$ and a transient class $U_0$. Given $U_d$ an ergodic class, there exists $n_d \geq 1$ and $\gamma_d \in [-1,1]$ such that
\begin{align}\label{ergodic}
\forall u\in U_d,\ \forall n\geq n_d,\ \E_u^\chi\left(\frac{1}{n}\sum_{m=1}^n f(u_m) \right) \geq  \gamma_d-\varepsilon,
\end{align}
and on the event $\left\{u_1 \in U_d\right\}$,
\begin{align}\label{convergence}
\frac{1}{n}\sum_{m=1}^n f(u_m) \text{ converges almost surely to } \gamma_d.
\end{align}
Moreover,  there exists $n_0 \geq 1$ such that 
\begin{align}\label{transient}
\forall u\in U,\ \forall n\geq n_0, \ \P^\chi_u(u_n \in U_0)\leq \varepsilon.
\end{align}
Let $l=\max_{d\in \{0,...,D\}} n_d$. The idea will be to decompose the payoff as a convex combination of evaluations that are constant on blocks of size $l$. Let $y_1\in \Delta(U)$.\\
We start by expressing the value for the $\liminf$ evaluation in terms of $(\gamma_d)_{d\in \{1,...,D\}}$. We know that
\begin{align*}
\E_{y_1}^{\chi}\left(\liminf_{n\rightarrow +\infty} \frac{1}{n}\sum_{m=1}^n f(u_m) \right)& =\E_{y_1}^{\chi}\left(\liminf_{n\rightarrow +\infty} \frac{1}{n}\sum_{m=l+1}^{n} f(u_m) \right),\\
& =\E_{y_1}^{\chi}\left(\E_{y_1}^{\chi}\left(\left.\liminf_{n\rightarrow +\infty} \frac{1}{n-l}\sum_{m=l+1}^{n} f(u_m)\right| u_{l+1} \right)\right).
\end{align*}

By Equation (\ref{transient}), we know that the probability that $u_{l+1}$ is in $U_0$ is smaller than $\varepsilon$. By Equation (\ref{convergence}), we know that, conditionally on $u_{l+1}$ to be in $U_d$, the payoff almost-surely converges to $\gamma_d$, hence we obtain
\begin{align}\label{tata0}
\E_{y_1}^{\chi}\left(\liminf_{n\rightarrow +\infty} \frac{1}{n}\sum_{m=1}^n f(u_m) \right)& \leq
\sum_{d=1}^D \gamma_d \P_{y_1}^\chi(u_{l+1} \in U_d)+\varepsilon.
\end{align}
\vspace{5mm}

Consider now $\theta=(\theta_m)_{m\geq 1}$ a history-dependent evaluation. We are going to establish the following lower bound, that is similar to the upper bound of Equation (\ref{tata0}):
\begin{align}\label{tata4}
\E_{y_1}^{\chi} \left( \sum_{m=1}^{+\infty} \theta_m f(u_m) \right) & \geq \sum_{d=1}^D (\gamma_d-\varepsilon)\P_{y_1}^\chi(u_{l+1} \in U_d)-4l I(\theta,y_1,\chi).
\end{align}
Combining Equations (\ref{tata0}) and (\ref{tata4}) implies Lemma \ref{titi}.
The rest of the section is dedicated to the proof of Equation (\ref{tata4}). Let us define a new evaluation $\omega=(\omega_m)_{m\geq 1}$, which is a piecewise-constant approximation of $(\theta_m)_{m \geq 1}$, in the following way: 
$\omega_m= \theta_{tl+1}$,
where $t$ is the unique integer such that $tl+1\leq m \leq (t+1)l$. 

\noindent First, we can bound the difference of the payoff under evaluation $\theta$ from stage $1$ and the payoff under evaluation $\omega$ from stage $l+1$ by a function depending on the irregularity.

\begin{lemma}\label{ecart}
For every function $b:H'_\infty \rightarrow [0,1]$, we have
\begin{align*}
\left|\E_{y_1}^{\chi} \left( \sum_{m=1}^{+\infty} \theta_m b(h'_m) \right)-\E_{y_1}^{\chi} \left( \sum_{m=l+1}^{+\infty} \omega_m b(h'_m) \right)\right|
& \leq 2l I(\theta,y_1,\chi).
\end{align*}
\end{lemma}

\begin{proof}
We have
\begin{align*}
& \left|\E_{y_1}^{\chi} \left( \sum_{m=1}^{+\infty} \theta_m b(h'_m) \right)-\E_{y_1}^{\chi} \left( \sum_{m=l+1}^{+\infty} \omega_m b(h'_m) \right)\right| \\
\hspace{0.5cm} & \leq \left|\E_{y_1}^{\chi} \left( \sum_{m=1}^{l} \theta_m b(h'_m) \right)\right|+\left|\E_{y_1}^{\chi} \left( \sum_{m=l+1}^{+\infty} \theta_m b(h'_m) \right)-\E_{y_1}^{\chi} \left( \sum_{m=l+1}^{+\infty} \omega_m b(h'_m) \right)\right|.&
\end{align*}

Let us focus on the left-hand term. For every $m\geq 1$, we have
\[
\E_{y_1}^{\chi} \left(|\theta_m|\right) \leq \E_{y_1}^{\chi}\left(|\theta_1|+\sum_{t=1}^{m-1} |\theta_{t+1}- \theta_{t}|\right) \leq I(\theta,y_1,\chi).
\]
Hence, 
\begin{align*}
\left|\E_{y_1}^{\chi} \left( \sum_{m=1}^{l} \theta_m b(h'_m) \right)\right| &\leq \E_{y_1}^{\chi} \left( \sum_{m=1}^{l} |\theta_m| \right)
 \leq l I(\theta,y_1,\chi). 
\end{align*}

The right-hand term is smaller than
\begin{align*}
 \E_{y_1}^{\chi} \left( \sum_{m=l+1}^{+\infty} |\theta_m-\omega_m |\right)
& =\sum_{t=1}^{+\infty} \E_{y_1}^{\chi} \left( \sum_{m=tl+1}^{(t+1)l} |\theta_m-\theta_{tl+1}| \right) \leq \sum_{t=1}^{+\infty} \E_{y_1}^{\chi} \left( l\sum_{m=tl+1}^{(t+1)l-1} |\theta_{m+1}-\theta_m |\right). 
 \end{align*}
 Therefore, the right-hand term is smaller than  $l I(\theta,y_1,\chi)$.
\end{proof}
\noindent Let us now give a lower bound on the payoff under evaluation $\omega$:
\begin{lemma}\label{lower}
We have
\begin{align*}
\E_{y_1}^{\chi} \left( \sum_{m=l+1}^{+\infty} \omega_m f(u_m) \right) \geq \E_{y_1}^{\chi} \left( \sum_{m=l+1}^{+\infty} \omega_m f'(h'_m) \right),
\end{align*}
where $f'$ is a function from histories to $[0,1]$ defined as follows:
\begin{align*}
\forall m \geq 1, \ \forall h'_m\in U^m,\ f'(h'_m)=\begin{cases} \gamma_d-\varepsilon & \text{ if } u_{l+1} \in U_d,\\
0 & \text{ otherwise }.
\end{cases}
\end{align*}
\end{lemma}

\begin{proof}
\noindent Indeed, we have
\begin{align}\label{tata1}
\E_{y_1}^{\chi} \left( \sum_{m=l+1}^{+\infty} \omega_m f(u_m) \right) & = \E_{y_1}^{\chi}\left(  \E_{y_1}^{\chi}\left(\left. \sum_{m=l+1}^{+\infty} \omega_m  f(u_m)\right| u_{l+1}  \right) \right).
\end{align}

\noindent Let us first focus on what happens on one block. Let $t\geq 0$, we have
\begin{align}
\E_{y_1}^{\chi} \left. \left(  \frac{1}{l}\sum_{m=tl+1}^{(t+1)l} f(u_m) \right| \mathcal{F}_{tl+1} \right) &=
\E_{u_{tl+1}}^{\chi}  \left( \frac{1}{l}\sum_{m=1}^{l} f(u_m) \right),\\
& \geq \begin{cases} \gamma_d-\varepsilon & \text{ on the event} \left\{u_{tl+1} \in U_d \right\},\\
0 & \text{ otherwise}.
\end{cases}
\end{align}

\noindent Moreover, when $u_{l+1}\in U_d$, the state stays forever in $U_d$ almost surely. Thus, one can replace $\left\{u_{tl+1} \in U_d \right\}$ by $\left\{u_{l+1} \in U_d \right\}$ in the above inequality. Hence,
for every $d\in \{1,...,D\}$ and every $t\geq 1$, we have
\begin{align*}
\E_{y_1}^{\chi}\left( \left(\sum_{m=tl+1}^{(t+1)l} \omega_m  f(u_m)\right) \1_{u_{l+1}\in U_d}\right)& \geq \E_{y_1}^{\chi}\left( \omega_{tl+1}l \left(\frac{1}{l}\sum_{m=tl+1}^{(t+1)l} f(u_m)\right) \1_{u_{l+1}\in U_d}\right),\\
& \geq \E_{y_1}^{\chi}\left( \omega_{tl+1}l (\gamma_d-\varepsilon) \1_{u_{l+1}\in U_d}\right),\\
& \geq \E_{y_1}^{\chi}\left( \sum_{m=tl+1}^{(t+1)l} \omega_{m} (\gamma_d-\varepsilon) \1_{u_{l+1}\in U_d}\right).
\end{align*}

\noindent We can now express the payoff under evaluation $\omega$:
\begin{align*}
\E_{y_1}^{\chi} \left( \sum_{m=l+1}^{+\infty} \omega_m f(u_m) \right) & = \sum_{d=0}^D \sum_{t=1}^{+\infty} \E_{y_1}^{\chi}\left( \left(\sum_{m=tl+1}^{(t+1)l} \omega_m  f(u_m)\right) \1_{u_{l+1}\in U_d}\right),\\
& \geq \sum_{d=1}^D \sum_{t=1}^{+\infty} \E_{y_1}^{\chi}\left( \sum_{m=tl+1}^{(t+1)l} \omega_{m} (\gamma_d-\varepsilon) \1_{u_{l+1}\in U_d}\right),\\
& \geq \E_{y_1}^{\chi}\left( \sum_{m=l+1}^{+\infty} \omega_{m} f'(h'_m) \right).
\end{align*}
\end{proof}
Then, we deduce from Lemma \ref{ecart} and Lemma \ref{lower} a lower bound for the payoff evaluated under the evaluation $(\theta_m)_{m\geq 1}$: we have 
\begin{align}
\E_{y_1}^{\chi} \left( \sum_{m=1}^{+\infty} \theta_m f(u_m) \right) & \geq \E_{y_1}^{\chi} \left( \sum_{m=l+1}^{+\infty} \omega_m f(u_m) \right)
-2l I(\theta,y_1,\chi),\\
& \geq \E_{y_1}^{\chi}\left( \sum_{m=l+1}^{+\infty} \omega_{m}f'(h'_m)\right)-2l I(\theta,y_1,\chi),\\
& \geq \E_{y_1}^{\chi}\left( \sum_{m=1}^{+\infty} \theta_{m}f'(h'_m)\right)-4l I(\theta,y_1,\chi),\\
& = \sum_{d=1}^D (\gamma_d-\varepsilon)\P^{y_1}_{\chi}(u_{l+1} \in U_d)-4l I(\theta,y_1,\chi). 
\end{align}


This concludes the proof of Lemma \ref{titi}.

\subsection{The decision-maker guarantees at most $v^*(x_1)$}

We now prove that the decision-maker can not guarantee more than $v^*(x_1)$. 

\begin{proposition}\label{limit_corollary_1}
Let $\Gamma$ be a POMDP and $(\theta^l)_{l\geq 1}$ be a sequence of normalized evaluation such that $(I(\theta^l,x_1))_{l\geq 1}$ converges to $0$. For every $x_1\in X$, we have
\begin{equation*}
\limsup_{l\rightarrow +\infty} v_{\theta^l}(x_1) \leq v^*(x_1).
\end{equation*}
\end{proposition}

If we could {\it a priori} restrict to strategies with finite memory, we would just adapt the argument of the previous section. Unfortunately, this is not the case, and thus we have to proceed differently. In order to prove Proposition \ref{limit_corollary_1}, we establish the following lemma. 

\begin{lemma}\label{limit_lemma_1}
Let $\Gamma$ be a POMDP, $x_1\in X$, $(\sigma^l)_{l\geq 1}$ a sequence of strategies and $(\theta^l)_{l\geq 1}$ a sequence of evaluations such that $\theta^l$ is normalized in expectation at $(x_1,\sigma^l)$, and $(I(\theta^l,x_1))_{l\geq 1}$ converges to $0$.
Then
\begin{equation*}
\limsup_{l\rightarrow +\infty} \gamma_{\theta^l}(x_1,\sigma^l) \leq v^*(x_1).
\end{equation*}
\end{lemma}

Proposition \ref{limit_corollary_1} is an immediate consequence of Lemma \ref{limit_lemma_1}. Indeed, for every $l\geq 1$, consider $\sigma^l$ a strategy $\frac{1}{l}$-optimal in the game $\Gamma_{\theta^l}(x_1)$. Since $\theta^l$ is a normalized evaluation, it is in particular normalized in expectation at $(x_1,\sigma^l)$. Therefore, one can apply Lemma \ref{limit_lemma_1} and obtain Proposition \ref{limit_corollary_1}. Proposition \ref{titi} and Proposition \ref{limit_corollary_1} together then yield Theorem \ref{uniform}.

The remainder of the section is dedicated to the proof of Lemma \ref{limit_lemma_1}. Its proof is decomposed in two steps. In this subsection, admitting that some measure $\mu^*$ is an invariant measure for the POMDP, we prove that $\underline{\underline{v}}_\infty(x_1)=v^*(x_1)$ is the maximal payoff that the decision-maker can guarantee. The next subsection is dedicated to the proof that $\mu^*$ is indeed an invariant measure (Lemma \ref{invariant}). 

Fix $\Gamma$, $x_1\in X$, $(\sigma^l)_{l\geq 1}$ a sequence of strategies and $(\theta^l)_{l\geq 1}$ a sequence of evaluations such that $\theta^l$ is normalized in expectation at $(x_1,\sigma^l)$, and $(I(\theta^l,x_1))_{l\geq 1}$ converges to $0$.
First, one can extract a subsequence of $(\sigma^l)_{l\geq 1}$ and ($\theta^l)_{l\geq 1}$ such that the payoff along the subsequence converges to $\limsup \gamma_{\theta^l}(x_1,\sigma^l)$. For convenience, we still denote these sequences by $(\sigma^l)_{l\geq 1}$ and $(\theta^l)_{l\geq 1}$. Denote by $\pi^l$ the measure $\P_{x_1}^{\sigma^l}$ on $H_\infty$.

              Let $\mu^l \in \Delta(X)$ be the barycenter of the law of $\sum_{m=1}^{+\infty} \theta_m^l \delta_{x_m}$, where $x_m$ is the belief of the decision-maker at stage $m$ along the observed history. Formally, let $\phi$ be the mapping from $H_\infty$ to $\Delta(X)$ defined by
\begin{align*}
\phi(h_\infty)=\sum_{m=1}^{+\infty}\theta^l_m(h_\infty)\delta_{x_m(h_\infty)}.
\end{align*}

Define $\mu^l$ as the projected image of $\pi^l$ by $\phi$: for any measurable function $f$ on $X$,
\begin{align*}
\int_{x\in X} f(x)\mu^l(dx) & = \int_{H_\infty} \left( \int_{X} f(x) \phi(h_\infty)(dx) \right) \pi^l(dh_{\infty}),\\
& = \int_{H_\infty} \left( \sum_{m=1}^{+\infty} \theta^l_m(h_\infty)f(x_m) \right) \pi^l(dh_\infty).
\end{align*}
Since $(\theta^l_m)_{m\geq 1}$ is normalized in expectation at $(x_1,\sigma^l)$, $\mu_l\in \Delta(X)$.

The state space $\Delta(X)$ being compact for the weak* topology, we can consider $\mu^*$ an accumulation point of the sequence $(\mu^l)_{l\geq 1}$. The following lemma is crucial, and its proof is done in the next subsection. 
\begin{lemma}\label{invariant}
$\mu^*$ is an invariant measure.
\end{lemma}

Let $\sigma$ such that $\mu^*$ is invariant under $\sigma$. As proved by the authors \cite{VZ16}, Birkhoff's theorem implies that under $\sigma$ and starting from $\mu^*$, the average payoff converges almost surely to a random variable with expectation $\widehat{v^*}(\mu^*)$. Hence, we have $\hat{g}(\mu^*) \leq \widehat{v^*}(\mu^*)$, and
\begin{align}
\gamma_{\theta^l}(x_1,\sigma^l) & = \E_{x_1}^{\sigma^l} \left( \sum_{m=1}^{+\infty} \theta^l_m r(k_m) \right)
 = \E_{x_1}^{\sigma^l} \left( \sum_{m=1}^{+\infty} \theta^l_m g(x_m) \right)
 = \hat{g}(\mu^l).
\end{align}
The function $g$ being $1$-Lipschitz on $(X,\|.\|_1)$, we deduce that $\hat{g}$ is $1$-Lipschitz on $(\Delta(X),d_{KR})$, hence continuous. Thus, by taking the limit when $l$ goes to infinity, we obtain that

\begin{equation*}
\limsup_{l\rightarrow +\infty} \gamma_{\theta^l}(x_1,\sigma^l) = \hat{g}(\mu^*).
\end{equation*}


We have $\m{E}_{x_1}^{\sigma^l}(v^*(x_{n+1}) | \mathcal{F}_n) \leq v^*(x_n)$. Moreover, since the payoff is bounded, $v^*$ is bounded. We will reformulate the payoff by introducing a randomized stopping-time in order to apply Doob optional stopping theorem. An interpretation of the payoff $\gamma_{\theta^l}(x_1,\sigma^l)$ can be the following: at every stage, there is a randomized variable that decides whether the process stops or continues; if the process stops, then the payoff is the current stage payoff whereas if the process continues, the payoff is $0$. 

Define the  extended space $\Omega=H_\infty \times [0,1]^{\N}$ and $\xi_l$ the extension of $\pi^l$ by i.i.d. uniform random variable on $[0,1].$ We then define the stopping time $\tau$ on $\Omega$ by
 \[
 \tau_l(h_\infty,x_1,...,x_n)=t  \text{ when for every } s<t,\ x_s \geq \frac{\theta^l_s}{1-\sum_{u=1}^{s-1} \theta^l_u} \text{ and } x_t \leq \frac{\theta^l_t}{1-\sum_{u=1}^{t-1} \theta^l_u}.
 \]
By definition of $\mu^l$, $\pi^l$ and $\xi_l$, we have
\begin{align*}
\widehat{v^*}(\mu^l)=\int_{X} v^*(x)\mu^l(dx)
              =\int_{H_\infty} \left(\sum_{m=1}^{+\infty} \theta^l_m v^*(x_m) \right)\pi^l(dh_\infty)
              =\int_{\Omega} v^*(x_{\tau_l}) \xi_l(d\omega).           
\end{align*}
The process $(v^*(x_n))$ is a super-martingale on the extended space $\Omega$, and is bounded. Thus, the optional stopping theorem yields
\begin{equation*}
\int_{\Omega} v^*(x_{\tau_l}) \xi_l(d\omega) \leq v^*(x_1).
\end{equation*}
Thus, setting $l$ to infinity, $\widehat{v^*}(\mu^*) \leq v^*(x_1)$. 
Hence,
\begin{equation}
\limsup_{l\rightarrow +\infty} \gamma_{\theta^l}(x_1,\sigma^l) = \hat{g}(\mu^*) \leq \widehat{v^*}(\mu^*) \leq v^*(x_1).
\end{equation}

This concludes the proof of Lemma \ref{limit_lemma_1}.\\

\subsection{Proof of Lemma \ref{invariant}}
The aim of this section is to prove that the measure $\mu^*$ defined in the previous section is an invariant measure. The proof is decomposed as follows. For every $l\geq 1$, we construct $\theta^l_f \in \Delta(H^o_f),\theta'^l_f \in \Delta(H^o_f)$ and $\mu^l,\mu'^l \in \Delta(X)$ such that 
\begin{enumerate}[label=(\roman*)]
\item \label{toto1}
\[
\|\theta'^l_f-\theta^l_f\|_1:=\sum_{h\in H^o_f} |\theta'^l_f(h)-\theta^l_f(h)| \leq 2I(\theta^l,x_1),
\]
\item \label{toto2}
\[
d_{KR}(\mu'^l,\mu^l) \leq \|\theta'^l_f-\theta^l_f\|_1,
\]
\item \label{toto3}$(\mu^l)_{l\geq 1}$ converges to $\mu^*$,
\item \label{toto4} There exists a strategy $\sigma_*^l$ such that $\mu'^l$ is the image of $\mu^l$ by $\sigma^l_*$.
\end{enumerate}

We can then combine these results to obtain the proof of Lemma \ref{invariant}. Indeed, by $\ref{toto3}$, $(\mu^l)_{l\geq 1}$ converges to $\mu^*$. Since $I(\theta^l,x_1)$ converges to $0$, $\ref{toto1}$ and $\ref{toto2}$ imply that $(\mu'^l)_{l\geq 1}$ also converges to $\mu^*$ when $l$ goes to infinity. The next step is to extend $\ref{toto4}$ to the limit and show the existence of a strategy $\sigma^*$ such that $\mu^*$ is the image of $\mu^*$ by $\sigma^*$. In order to obtain this result, we apply Proposition 20 and Proposition 32 in \cite{VZ16}. Hence, we obtain that $\mu^*$ is an invariant measure. 



\subsubsection{Definition of $\theta^l_f$ and $\theta'^l_f$}

We first define $\theta^l_f$.  We consider the mapping $\psi$ from $H_\infty$ to $\Delta(H^o_f)$ such that
\begin{align*}
\psi(h_\infty)=\sum_{m=1}^{+\infty} \theta^l_m(h_\infty)\delta_{h^o_m},
\end{align*}
where $h^o_m$ is the truncation of $h_\infty$ up to stage $m$ restricted to actions and signals. Define $\theta^l_f$ to be the projected image of $\pi^l$ by $\psi$. \\

Informally, $\theta'^l_f$ is defined in order to fit the following story: an observed history $h$ is chosen randomly following $\theta^l_f$, and told to the decision-maker. He then plays $\sigma^l(h)$, yielding a new observed history of length $|h|+1$. Formally, let $\xi$ be the mapping from the set of observed histories to $X$ that associates to an observed history $h_n$ the belief of the last stage $x_n$.
Denote by $\theta'^l_f$ the distribution such that
\[
\forall (h,i,s)\in H^o_f\times I \times S, \ \theta'^l_f(h,i,s)=\theta^l_f(h)q(\xi(h),i)(s)\sigma^l(h)(i).
\]

We now prove that the distance between $\theta^l_f$ and $\theta'^l_f$ is controlled by the irregularity.
\begin{lemma}
We have
\[
\|\theta^l_f-\theta'^l_f\|_1 \leq 2 I(\theta^l,x_1).
\]
\end{lemma}

\begin{proof}
For the proof, it will be convenient to decompose histories according to their length. For every $m\geq 1$, denote by $\pi^l_m$ the image of $\pi^l$ by the projection on the histories of length $m$. Let $m\geq 1$, $h'_m\in H^o_m$ and $(i,s)\in I \times S$, then we have
\begin{align*}
\theta^l_f((h'_m,i,s))&=\int_{h_\infty\in H_\infty} \left( \theta_{m+1}^l(h_\infty)\1_{h_\infty\left|_{m+1}\right.=(h'_m,i,s)}\right)\pi^l(dh_\infty), \\
& = \pi^l_{m+1}((h'_m,i,s))\theta^l_{m+1}((h'_m,i,s)),
\end{align*}
since $\theta^l_{m+1}$ is $\mathcal{F}^o_{m+1}$-measurable. Moreover, we have
\begin{align*}
\theta'^l_f((h'_m,i,s))&=\theta^l_f(h'_m)q(\xi(h'_m),i)(s)\sigma^l(h'_m)(i),\\
& = \pi^l_m(h'_m)\theta_m^l(h'_m)q(\xi(h'_m),i)(s)\sigma^l(h'_m)(i),\\
& = \pi^l_{m+1}((h'_m,i,s))\theta_m^l(h'_m),
\end{align*}
since 
\[
\pi^l_{m+1}((h'_m,i,s))=\pi^l_m(h'_m)\sigma^l(h'_m)(i)q(\xi(h'_m),i)(s).
\]
It follows that 
\begin{align*}
&\sum_{(h'_m,i,s)\in H^o_m\times I  \times S} |\theta'^l_f((h'_m,i,s))-\theta^l_f((h'_m,i,s))| 
\\
& = \sum_{(h'_m,i,s)\in H^o_m\times I  \times S} \pi^l_{m+1}((h'_m,i,s))|\theta_m^l(h'_m)-\theta^l_{m+1}((h'_m,i,s))|,\\
&= \int_{h_\infty \in H_\infty} |\theta_m^l(h_\infty)-\theta^l_{m+1}(h_\infty)|\pi^l(dh_\infty).
\end{align*}
The result follows by summation over $m$.
\end{proof}

\subsubsection{Link with $\mu^l$ and definition of $\mu'^l$}

We define $\mu'^l$ to be the image of $\theta'^l_f$ by $\xi$. Moreover, we have the following results.

\begin{lemma}
\textcolor{white}{}
\begin{itemize}
\item $\mu^l$ is the image of $\theta^l_f$ by $\xi$,
\item $d_{KR}(\mu^l,\mu'^l) \leq \|\theta^l_f-\theta'^l_f\|_1$.
\end{itemize}
\end{lemma}

\begin{proof}
\underline{$\theta^l_f$ and $\mu^l$:} By definition, $\theta^l_f$ satisfies that for every measurable function $f: H^0_f \rightarrow \m{R}$, we have 
\begin{align*}
\int_{H^o_f} f(h)\theta^l_f(dh) & = \int_{H_\infty} \left(\sum_{m=1}^{+\infty} \theta_m^l(h_\infty)f(h_m) \right) \pi^l(dh_\infty).
\end{align*}
Denote by $\nu$ the image of $\theta^l_f$ by $\xi$, then for every $f: X \rightarrow \m{R}$ measurable, we have
\begin{align*}
\int_{X} f(x)\nu(dx)  =\int_{H^o_f} f(\xi(h)) \theta^l_f(dh)
= \int_{H_\infty} \left(\sum_{m=1}^{+\infty} \theta_m^l(h_\infty)f(x_m) \right) \pi^l(dh_\infty).
\end{align*}
Hence, $\nu$ is the projected image of $\pi^l$ by $\phi$ and, by uniqueness, $\nu$ is equal to $\mu^l$.\\

\underline{$d_{KR}(\mu^l,\mu'^l) \leq \|\theta^l_f-\theta'^l_f\|_1$:}
The $L^1$ norm can be seen as the Kantorovitch-Rubinstein metric associated to the discrete distance $d$ on $H_f$. Hence, if $f$ is $1$-Lipschitz from $X$ with $L^1$-norm to $[-1,1]$, then the mapping $f \circ \xi$ is $1$-Lipschitz for $d$ since
\[
|f(\xi(h))-f(\xi(h'))|\leq \|\xi(h)-\xi(h')\|_1 \leq d(h,h').
\]
It follows that
\begin{align*}
\int_{x\in X} f(x) \mu^l(dx)-\int_{x\in X} f(x) \mu'^l(dx) & =\int_{h\in H^o_f} f(\xi(h)) \theta^l_f(dh)-\int_{h\in H^o_f} f(\xi(h)) \theta'^l_f(dh),\\
& \leq \|\theta^l_f-\theta'^l_f\|_1.
\end{align*}
Since it is true for any $1$-Lipschitz function $f$, we obtain $d_{KR}(\mu^l,\mu'^l)\leq \|\theta^l_f-\theta'^l_f\|_1$.
\end{proof} 

\subsubsection{Definition of the strategy $\sigma^*$}

We now construct the strategy $\sigma^l_*$ that maps $\mu^l$ to $\mu'^l$. Consider the distribution $\theta^l_f$, we can define the extended measure denoted $\overline{\theta^l_f}$ over $H^o_f\times X$ by associating to an observed history $h$, its end-belief $\xi(h)$.
We can then consider the disintegration with respect to the last coordinate. There exists a kernel $\mathcal{K}: X \times \mathcal{B}(H^o_f) \rightarrow [0,1]$ such that
for all $f: H^0_f \times X \rightarrow \m{R}$ measurable, 
\begin{align*}
\int_{(h,x)\in H^o_f\times X} f(h,x) \overline{\theta^l_f}(dh,dx)=\int_{x \in X} \left( \int_{h \in H^o_f} f(h,x) \mathcal{K}(x,dh) \right) \mu^l(dx).
\end{align*}
Informally, the kernel $\mathcal{K}(y,dh)$ associates to the belief $y$ a distribution over finite observed histories that have $y$ as end-belief. We now consider for every $y\in X$, $\sigma^l_*(y)$ to be the projected image of $\mathcal{K}(y,.)$ by the mapping $\sigma^l: H^o_f \rightarrow \Delta(I)$ that associates to any finite history the distribution over the actions played.

\begin{lemma}
The image of $\mu^l$ by $\sigma^l_*$ is $\mu'^l$.
\end{lemma}

\begin{proof}
Recall that $\mu'^l$ is the image of $\theta'^l_f$ by $\xi$: hence, it is the unique measure such that for every measurable function on $X$,
\begin{align*}
\int_{x\in X} f(x) \mu'^l(dx) 
&= \int_{(h,i,s) \in H^o_f\times I \times S} f(\xi(h,i,s)))\theta'^l_f(d(h,i,s)),\\
&= \sum_{(i,s)\in I\times S} \left( \int_{h \in H^o_f} f(\xi(h,i,s))q(\xi(h),i)(s)\sigma^l(h)(i)\theta^l_f(dh)\right).
\end{align*}
Define $\nu$ as the image of $\mu^l$ by the strategy $\sigma_*^l$ and check that it satisfies the characterization of $\mu'^l$. 
Let $f:X \rightarrow \m{R}$ measurable. By definition of the image of $\mu^l$ by the strategy $\sigma^l_*$, 
\begin{align*}
\int_{y\in X} f(y)\nu(dy) & =\int_{x\in X} \left( \int_{y\in X} f(y) (\sigma^l_*)^{\sharp q}(x)(dy) \right) \mu^l(dx),\\
& =\int_{x\in X} \left( \sum_{(k',i,s)\in K\times I \times S} f(\overline{q}(x,i,s)) \sigma^l_*(x)(i)q(k',i)(s)x(k') \right) \mu^l(dx),\\
& =\sum_{(i,s)\in I \times S} \left( \int_{x\in X} f(\overline{q}(x,i,s)) \sigma^l_*(x)(i)q(x,i)(s) \mu^l(dx)\right).
\end{align*}
It follows by definition of $\sigma^l_*$ as the projected image of $\mathcal{K}(x,dh)$ by $\sigma^l$ that
\begin{align*}
\int_{y\in X} f(y)\nu(dy) & =\sum_{(i,s)\in I \times S} \left( \int_{x\in X} \left(\int_{h\in H^o_f} f(\overline{q}(x,i,s)) \sigma^l(h)(i)q(x,i)(s) \mathcal{K}(x,dh) \right) \mu^l(dx)\right),\\
& =\sum_{(i,s)\in I \times S} \left( \int_{(h,x)\in H^o_f\times X} f(\overline{q}(x,i,s)) \sigma^l(h)(i)q(x,i)(s) \overline{\theta^l_f}(dh,dx) \right),\\
& =\sum_{(i,s)\in I \times S} \left( \int_{h\in H^o_f} f(\overline{q}(\xi(h),i,s)) \sigma^l(h)(i)q(\xi(h),i)(s) \theta^l_f(dh) \right).
\end{align*}

By Bayes rule, we have
$
\overline{q}(\xi(h),i,s)=\xi(h,i,s).
$
It implies that
\begin{align*}
\int_{y\in X} f(y)\nu(dy) & =\sum_{(i,s)\in I \times S} \left( \int_{h\in H^o_f} f(\xi(h,i,s)) q(\xi(h),i)(s)\sigma^l(h)(i)\theta^l_f(dh) \right).
\end{align*}

\noindent We recognize the characterization of $\mu'_l$, and the lemma is proved. 
\end{proof}

\section{Proof of Theorem \ref{limsup}}

In this section, we establish the proof of Theorem \ref{limsup}. First, we replace Proposition \ref{titi} by a direct comparison between the $\liminf$ payoff and the $\limsup$-belief payoff. 
Second, we introduce an appropriate sequence of evaluations and apply Lemma \ref{limit_lemma_1} to prove the other inequality.

\subsection{A history-dependent evaluation that approximates the $\limsup$-evaluation}

Fix $x_1\in X$. As we have already seen, it is straightforward that
$\underline{\underline{v}}_\infty(x_1) \leq \overline{v}_\infty(x_1)$.
Let us prove the reverse inequality. Let $\varepsilon>0$ and $\sigma$ be an $\varepsilon$-optimal strategy in $\overline{\Gamma}_\infty(x_1)$. We denote by $\pi$ the probability distribution generated by $x_1$ and $\sigma$ on $(K\times I \times S)^\infty$, previously denoted by $\m{P}_{x_1}^{\sigma}$.
Let $l \geq 1$. We define the integer random variable $\eta^l$:
 \begin{equation}
\eta^l:=\inf \left\{n' \geq l \ | \ \frac{1}{n'} \sum_{m=1}^{n'} r(x_m)  \geq \limsup_{N \rightarrow+\infty} \frac{1}{N} \sum_{m=1}^N r(x_m)-\frac{1}{l} \right\}.
\end{equation}
By construction, $\eta^l$ is $\mathcal{F}^o$-measurable. Define the sequence of functions $(\theta^l_m)_{m\geq 1}$ such that for every $h_\infty \in H_\infty$ and for every $m\geq 1$,
\begin{align*}
\theta^l_m(h_\infty)& := \frac{1}{\eta^l(h_\infty)}\1_{m\leq \eta^l(h_\infty)}. 
\end{align*}
Notice that the evaluation $(\theta^l_m)_{m\geq 1}$ is not history-dependent. Nevertheless, one can  define 
$\rho^l_m=\E^{\sigma}_{x_1}\left( \theta^l_m |\mathcal{F}^o_m\right)$.
By construction, $(\rho^l_m)_{m\geq 1}$ is history-dependent and normalized in expectation at $(x_1,\sigma)$ and we have
\begin{align*}
\overline{v}_\infty(x_1) \leq \E^{\sigma}_{x_1} \left( \limsup_{N\rightarrow +\infty} \frac{1}{N} \sum_{m=1}^{+\infty} r(x_m) \right) +\varepsilon
\leq  \gamma_{\theta^l}(x_1,\sigma)+\varepsilon+\frac{1}{l}
= \gamma_{\rho^l}(x_1,\sigma)+\varepsilon+\frac{1}{l}. 
\end{align*}
    
\noindent Moreover, we can control the irregularity of the sequence $(\rho^l)_{l \geq 1}$.

\begin{proposition}\label{irregularity}
The sequence $(I(\rho^l,x_1))_{l\geq 1}$ goes to $0$ when $l$ goes to infinity.
\end{proposition}
Assume that this proposition holds. By  Lemma \ref{limit_lemma_1}, we have
$\limsup_{l \rightarrow +\infty} \gamma_{\rho^l}(x_1,\sigma) \leq \underline{\underline{v}}_\infty(x_1)$.
Thus, $\overline{v}_\infty(x_1) \leq \underline{\underline{v}}_\infty(x_1)+\varepsilon$, and since $\varepsilon$ is arbitrary, $\overline{v}_\infty(x_1) \leq \underline{\underline{v}}_\infty(x_1)$, and Theorem \ref{limsup} is proved. 

\subsection{Control of the irregularity: Proof of Proposition \ref{irregularity}}

In order to finish the proof of Theorem \ref{limsup}, we show that the irregularity of $\rho^l$ goes to $0$ when $l$ goes to infinity. It relies on the fact that the sequence $(\rho^l_m)_{m\geq 1}$ is a super-martingale such that for every $m\geq 1$, 
\begin{align}\label{majoration}
|\rho^l_m|\leq \min\left(\frac{1}{l},\frac{1}{m}\right).
\end{align}

Let $n \geq 1$. By definition of $\rho^l_n$, one has
\begin{align*}
\m{E}_{x_1}^{\sigma} \left(\rho^l_{n+1} \left| \mathcal{F}_n \right. \right) &  = \m{E}_{x_1}^{\sigma} \left(\m{E}_{x_1}^{\sigma}\left( \frac{1}{\eta^l} \1_{n+1\leq \eta^l} \left| \mathcal{F}_{n+1} \right. \right) \left| \mathcal{F}_n \right. \right),\\
& = \m{E}_{x_1}^{\sigma} \left( \frac{1}{\eta^l} \1_{n\leq \eta^l} \left| \mathcal{F}_n \right. \right) - \m{E}_{x_1}^{\sigma} \left( \frac{1}{\eta^l} \1_{n=\eta^l} \left| \mathcal{F}_n \right. \right),\\
& = \rho^l_n-\frac{1}{n} \m{E}_{x_1}^{\sigma} \left( \1_{n=\eta^l} \left| \mathcal{F}_n \right. \right).
\end{align*}
since $\{n \leq \eta^l\}=\{n <\eta^l\} \cup\{n=\eta^l\}=\{n+1\leq \eta^l\} \cup \{n=\eta^l\}$.%
The previous computation implies the following equality for the square difference: for every stage $n\geq 1$,
\begin{align*}
\m{E}_{x_1}^{\sigma} \left( \left(\rho^l_{n+1}-\rho^l_n \right)^2 \right)  
 & = \m{E}_{x_1}^{\sigma} \left( (\rho^l_{n+1})^2 \right) -2\m{E}_{x_1}^{\sigma}\left( \rho^l_n \m{E}_{x_1}^{\sigma}\left( \rho^l_{n+1}\left| \mathcal{F}_n \right. \right)\right)+ \m{E}_{x_1}^{\sigma} \left( (\rho^l_n)^2 \right), \\
 & = \m{E}_{x_1}^{\sigma} \left( (\rho^l_{n+1})^2 \right) -2\m{E}_{x_1}^{\sigma}\left( (\rho^l_n)^2- \frac{\rho^l_n}{n} \m{E}_{x_1}^{\sigma}\left( 1_{n=\eta^l} \left| \mathcal{F}_n \right. \right)\right)+ \m{E}_{x_1}^{\sigma} \left( (\rho^l_n)^2 \right),\\
 &= \m{E}_{x_1}^{\sigma} \left( (\rho^l_{n+1})^2  - (\rho^l_n)^2 \right)+2 \m{E}_{x_1}^{\sigma}\left(\frac{\rho^l_n}{n} 1_{n=\eta^l}\right).
\end{align*}

\begin{proposition} \label{imp_limsup}
$I(\theta^l,x_1)$ converges to $0$ when $l$ goes to infinity.
\end{proposition}

\begin{proof}
We first recall a classical result on martingales, obtained by using Cauchy-Schwartz inequality:
\begin{align}\label{majoration_block_m}
\m{E}_{x_1}^{\sigma}\left( \sum_{t=m}^n |\rho^l_{t+1}-\rho^l_t | \right) & \leq \sqrt{ \m{E}_{x_1}^{\sigma}\left( \sum_{t=m}^n (\rho^l_{t+1}-\rho^l_t )^2 \right) \m{E}_{x_1}^{\sigma}\left( \sum_{t=m}^n 1 \right)}, \\
& \leq \sqrt{n-m+1} \sqrt{ \m{E}_{x_1}^{\sigma}\left( (\rho^l_{n+1})^2-(\rho^l_{m})^2\right)+ 2 \sum_{t=m}^{n} \m{E}_{x_1}^{\sigma}\left(\frac{\rho^l_t}{t} 1_{t=\eta^l} \right)},\\
& \leq \sqrt{n-m+1} \sqrt{ \m{E}_{x_1}^{\sigma}\left( \frac{1}{m^2}+\frac{1}{m^2}+2\frac{1}{m^2} \right)},\\
& \leq 2\frac{\sqrt{n-m+1}}{m}. \label{majoration_block_m:5}
\end{align}
where the second inequality is deduced from the previous computation and a telescopic summation. The third inequality stems from the fact that, by assumption, $|\theta^l_{t}|$ is bounded from above on $\left\{m, \dots, n \right\}$ by $\frac{1}{m}$ (Equation (\ref{majoration})).
 Since $|\theta^l_{t}|$ is also bounded from above by $\frac{1}{l}$, Equation (\ref{majoration}) and similar computations yield 
\begin{align}\label{majoration_block_l}
\m{E}_{x_1}^{\sigma}\left( \sum_{t=m}^n |\rho^l_{t+1}-\rho^l_t | \right) & \leq 2\frac{\sqrt{n-m+1}}{l}.
\end{align}

We now split the set of integers into a sequence of blocks $B_k$ such that $B_k$ has length $k$ and then apply on each block the previous equations. The sequence of blocks is defined as follows: $B_1=\left\{\rho^l_1 \right\}$,
and for every $k\geq 2$, $B_k=\left\{t_{k-1}+1,\dots,t_k \right\}$, where $t_{k}=k(k+1)/2$.
\\
On each $B_k$, we apply inequality (\ref{majoration_block_m:5}) and obtain
\begin{align*}
\m{E}_{x_1}^{\sigma}\left( \sum_{t=t_{k-1}+1}^{t_{k}} |\rho^l_{t+1}-\rho^l_t | \right) & \leq \frac{2\sqrt{k}}{  (t_{k-1}+1)}
 = \frac{4\sqrt{k}}{k(k-1)+2}.
\end{align*}
Let $a_k= \frac{4\sqrt{k}}{k(k-1)+2}$, $k \geq 1$. The serie $(a_k)_{k\geq 1}$ is converging,  hence $I(\theta^l,\sigma,x_1)$ is finite and moreover for every $m\geq 1$,
\[
\m{E}_{x_1}^{\sigma}\left(\sum_{t=m}^{+\infty} |\rho^l_{t+1}-\rho^l_t | \right) \leq \sum_{t=\phi(m)}^{+\infty} a_t,
\]
for some $\phi(m)$ such that $\phi(m)$ is increasing to $+\infty$. Using that $|\theta^l_{t}|$ is bounded by above over the interval $1$ to $l-1$ by $\frac{1}{l}$, we deduce that
\begin{align*}
I(\theta^l,x_1,\sigma) & \leq \m{E}_{x_1}^{\sigma} \left( \rho^l_1+ \sum_{t=1}^{l-1} |\rho^l_{t+1}-\rho^l_t | + \sum_{t=l}^{+\infty} |\rho^l_{t+1}-\rho^l_t | \right),\\
& \leq \frac{1}{l}+\frac{2\sqrt{l-1}}{l}+ \sum_{t=\phi(l)}^{+\infty} a_t,
\end{align*}
where we apply inequality (\ref{majoration_block_l}) to the block $\left\{1,\dots, l-1\right\}$ and inequality (\ref{majoration_block_m:5}) to the other blocks. 
The three terms converge to $0$ as $l$ goes to infinity, and Proposition \ref{imp_limsup} is proved.
\end{proof}

\section*{Acknowledgments}
Venel gratefully  acknowledges   the support of the Agence Nationale de la Recherche, under grant ANR CIGNE, ANR-15-CE38-0007.
\\
This research benefited from the support of the FMJH Program PGMO ``Regularization of stochastic games" and from the support
of EDF, Thales, Orange and Criteo.

\begin{small}
\bibliography{pomdp}
\end{small}

\end{document}